\documentclass[12pt]{article}

\usepackage{setspace}
\usepackage{amssymb}
\usepackage{amsmath}
\usepackage{amsthm}
\usepackage{mathrsfs}
\usepackage{enumerate}
\usepackage{color}
\usepackage[all]{xy}
\usepackage[left=2cm, right=2cm, top=3cm]{geometry}
\usepackage{changepage}
\usepackage{blkarray}
\usepackage{tikz}
\usepackage{xcolor}
\usepackage{tcolorbox}
\usepackage{graphicx}

\usepackage[pdfusetitle]{hyperref}
\hypersetup{
    colorlinks=false,  
    linkcolor=black,   
    citecolor=black,   
    filecolor=black,   
    urlcolor=black,    
    }

\newtheorem{theorem}{Theorem}[section]
\newtheorem{prop}[theorem]{Proposition}
\newtheorem{cor}[theorem]{Corollary}
\newtheorem{lemma}[theorem]{Lemma}

\theoremstyle{definition}
\newtheorem{example}[theorem]{Example}
\newtheorem{Definition}[theorem]{Definition}

\theoremstyle{remark}
\newtheorem{remark}[theorem]{Remark}

\newcommand{\overbar}[1]{\mkern 1.5mu\overline{\mkern-1.5mu#1\mkern-1.5mu}\mkern 1.5mu}

\def\cok{{\mathrm{cok\,}}}

\def\id{{\mathrm{id}}}

\begin{document}
\title{\vspace{-2cm} A characterisation for the category of Hilbert spaces}
\author{Stephen Lack \\
  steve.lack@mq.edu.au \\
  Shay Tobin \\
  shay.tobin@mq.edu.au \\
  \\
  School of Mathematical and Physical Sciences\\
  Macquarie University NSW 2109 \\
  Australia}
\maketitle
\noindent

\begin{abstract}
The categories of real and of complex Hilbert spaces with bounded linear maps have received purely categorical characterisations by Chris Heunen and Andre Kornell. These characterisations are achieved through Solèr's theorem, a result which shows that certain orthomodularity conditions on a Hermitian space over an involutive division ring result in a Hilbert space with the division ring being either the reals, complexes or quarternions. The characterisation by Heunen and Kornell makes use of a monoidal structure, which in turn excludes the category of quarternionic Hilbert spaces. We provide an alternative characterisation without the assumption of monoidal structure on the category. This new approach not only gives a new characterisation of the categories of real and of complex Hilbert spaces, but also the category of quaternionic Hilbert spaces.
\end{abstract}

\section{Introduction}

We consider the category $\mathbf{Hilb}_{\mathbb{K}}$ of Hilbert
spaces over $\mathbb{K}$, where $\mathbb{K}$ is either the real
numbers $\mathbb{R}$, the complex numbers $\mathbb{C}$, or the
quaternions $\mathbb{H}$. In each case, the morphisms are the bounded
linear maps. In 2022, Chris Heunen and Andre Kornell characterised
$\mathbf{Hilb}_\mathbb{K}$ as a dagger monoidal category where
  $\mathbb{K}$ is $\mathbb{R}$ or $\mathbb{C}$ by establishing a dagger monoidal
equivalence between $\mathbf{Hilb}_\mathbb{K}$ and a category
$\mathbf{C}$ which satisfied a list of purely categorical axioms
\cite{Heunen2022}. The category $\mathbf{C}$ was equipped with a
dagger monoidal structure that defined a commutative multiplication on
the scalars, whence the restriction to $\mathbb{R}$ or
  $\mathbb{C}$. In this paper we follow a similar approach by building a dagger equivalence between $\mathbf{C}$ and $\mathbf{Hilb}$ but do not assume dagger monoidal structure on $\mathbf{C}$. As well as giving a new characterisation of $\mathbf{Hilb}_{\mathbb{R}}$ and $\mathbf{Hilb}_{\mathbb{C}}$, it also allows us to treat $\mathbf{Hilb}_{\mathbb{H}}$.

In 1995, Maria Pia Solèr developed a characterisation of Hilbert
spaces using orthomodular spaces. This was inspired by results in quantum logic by Birkhoff and von Neumann in the mid 1930s \cite{Neumann1936}, and built on the work by Kaplansky on infinite dimensional quadratic forms in the 1950s \cite{Kaplansky1953} and Piron's representation theorem which proved a correspondence between propositional systems and orthomodular spaces in the 1970s \cite{Piron1964}. Solèr's theorem was originally stated as follows:

\begin{theorem}[Solèr's Theorem \cite{Soler1995}]\label{solerstheorem}
Let $(\mathcal{H},\langle\cdot,\cdot\rangle)$ be an infinite dimensional orthomodular space over an involutive division ring $\mathbb{K}$  which contains an orthonormal sequence $(x_i)_{i\in \mathbb{N}}$. Then $\mathbb{K}$ is either $\mathbb{R}$, $\mathbb{C}$ or $\mathbb{H}$, and $(\mathcal{H},\langle\cdot,\cdot\rangle)$ is a Hilbert space over $\mathbb{K}$.
\end{theorem}
\noindent
We follow \cite{Heunen2022} in using Solèr's theorem in our characterisation of $\mathbf{Hilb}_\mathbb{K}$. Much of what we do follows the approach in \cite{Heunen2022}. Since our assumptions are different - especially in not assuming a tensor product - the details necessarily differ.
  
To establish an equivalence between an abstract category $\mathbf{C}$
and $\mathbf{Hilb}_\mathbb{K}$ we begin with the fact that any Hilbert
space $\mathcal{H}$ is in bijection to the set of bounded linear maps
$\mathbb{K}\to \mathcal{H}$, and in particular the scalars
$\mathbb{K}$ correspond to bounded linear maps
$\mathbb{K}\to\mathbb{K}$. This correspondence allows us to treat
homsets of the form $\mathbf{C}(K,H)$ for some $K$, as candidates for
Hilbert spaces with $\mathbf{C}(K,K)$ as scalars. We eventually lift
the hom-functor $\mathbf{C}(K,-)$ to a dagger equivalence $\mathbf{C}\simeq \mathbf{Hilb}_{\mathbf{C}(K,K)}$. 

In Section 2 we set $\mathbf{C}$ to  be a dagger category and study
the linear structure that emerges on $\mathbf{C}(K,H)$ when equipping
$\mathbf{C}$ with dagger biproducts. In Section 3, we require
$\mathbf{C}$ to have dagger equalizers, all dagger monomorphisms in $\mathbf{C}$ to be kernels, and the object $K$ to be a simple generator. The scalars $\mathbf{C}(K,K)$ then have the
structure of a division ring with involution. In Section 4, we suppose the wide subcategory of dagger monomorphisms of $\mathbf{C}$ to have directed colimits and translate the notion of orthogonal decomposition for Hilbert spaces to $\mathbf{C}(K,H)$ so that each such homset is an orthomodular space. In Section 5 we introduce the notion of an orthonormal basis for $\mathbf{C}(K,H)$. This condition allows the construction of colimit objects that directly correspond to the familiar Hilbert spaces  $\ell^2(X)$ of a set $X$. Finally, in Section 6 we verify that $\mathbf{C}({K,H})$ satisfies the conditions for Solèr's Theorem and show $\mathbf{C}(K,-)\colon \mathbf{C}\to\mathbf{Hilb}_{\mathbf{C}(K,K)}$ to be a dagger equivalence. This motivates the following definition which will form the basis of our characterisation: 
\begin{Definition}[Hilbert category]\label{defhilbertcategory}
A \textit{Hilbert category}  $\mathbf{C}$ is a category satisfying the following conditions:
\begin{itemize}
\item[(D)]  $\mathbf{C}$ is equipped with a dagger $\dagger\colon \mathbf{C}^\mathrm{op}\to \mathbf{C}$;
\item[(G)]  $\mathbf{C}$ is equipped with a simple generator $K$;
\item[(B)]  $\mathbf{C}$ has a zero object and any pair of objects $A,B\in \mathbf{C}$ has a dagger biproduct;
\item[(E)]  Any pair of parallel morphisms $f,g$ has a dagger equalizer;
\item[(K)]  Any dagger monomorphism is a dagger kernel;
\item[(C)]  The wide subcategory $\mathbf{C}_{\mathrm{dm}}$ of dagger monomorphisms has directed colimits;
\end{itemize}
\end{Definition}
\noindent In Section 7 we briefly investigate the effects of a monoidal structure on a Hilbert category and the connection between our result and the main result of \cite{Heunen2022}.

The material in this paper is based on the second-named author's 2023 Master's thesis \cite{Tobin2024}. During the conversion of this thesis into
the present format we became aware of the preprint
\cite{Dimeglio2024}, which among other things, gives an alternative
approach to the characterisation in \cite{Heunen2022} which does not
rely on Sol\`er's theorem.

\section{Linear Structure}

We now look at the categorical structures required on a category
$\mathbf{C}$ so that for a given $K$ the homset $\mathbf{C}(K,K)$ is
an involutive semiring and for any object $H$ the homset
$\mathbf{C}(K,H)$ is a right semimodule over $\mathbf{C}(K,K)$. We
speak of semirings and semimodules instead of ring and modules when additive inverses are not assumed to exist.

\begin{Definition}[Dagger category]
A \textit{dagger} on a category $\mathbf{C}$ is a contravariant involution $\dagger\colon \mathbf{C}^{\mathrm{op}}\to \mathbf{C}$ which acts as the identity on objects; thus for each $f\colon A\to B$ there is a chosen map $f^\dagger\colon B\to A$. A \textit{dagger category} is a category equipped with a dagger.
\end{Definition}

\begin{Definition}[Dagger morphisms]
Let $f\colon A\to B$ be a morphism of a dagger category $\mathbf{C}$. We say that $f$ is a \textit{dagger monomorphism} if $f^\dagger f = 1_A$ and a \textit{dagger epimorphism} if $ff^\dagger=1_B$. If $f$ is both a $\dagger$-mono and $\dagger$-epi then we call it a \textit{unitary isomorphism}; then $f$ is invertible with $f^\dagger=f^{-1}$.
\end{Definition}

\begin{Definition}[Dagger Functor]
A \textit{dagger functor} is a functor $F\colon \mathbf{C}\to \mathbf{D}$ between dagger categories $(\mathbf{C},\dagger$) and $(\mathbf{D},\dagger )$ which preserves the dagger structure: $F(f^\dagger) = F(f)^\dagger$ for each $f\in \mathbf{C}(A,B)$. 
\end{Definition}

\begin{Definition}[Dagger Biproduct]
Let $\mathbf{C}$ be a dagger category with zero objects. A \textit{dagger biproduct} between two objects $A$ and $B$ in $\mathbf{C}$ is a product of $A$ and $B$, denoted $A\oplus B$, with projections $p\colon A\oplus B\to A$ and $q\colon A\oplus B\to B$ where $p$ and $q$ are dagger epimorphisms, and $pq^\dagger$ and $qp^\dagger$ are zero morphisms. 
\end{Definition}

Equipping the category $\mathbf{C}$ with a (dagger) biproduct gives an enrichment of $\mathbf{C}$ in commutative monoids: each homset $\mathbf{C}(A,B)$ is a commutative monoid under an addition defined as
\begin{align*}
+\colon \mathbf{C}(A,B)\times\mathbf{C}(A,B)&\to \mathbf{C}(A,B)\\
(f,g)&\mapsto f+g := \nabla(f\oplus g)\Delta
\end{align*}
where $\Delta\colon A\to A\oplus A$ and $\nabla\colon B\oplus B\to B$ are the diagonal and codiagonal maps respectively.

The unit for this monoid is the zero map $0\colon A\to B$. It follows that $\langle \mathbf{C}(K,K),+, \circ \rangle $ is a semiring. 
We define a right action of $\mathbf{C}(K,K)$ on $\mathbf{C}(K,A)$ by
\begin{align*}
\cdot\colon \mathbf{C}(K,A)\times\mathbf{C}(K,K)&\to \mathbf{C}(K,A)\\
(f,\lambda)&\mapsto f\cdot \lambda :=f\circ \lambda
\end{align*}
which along with the biproduct structure on $\mathbf{C}$ leads to the following lemma.

\begin{lemma}\label{semimodule}
Let $\mathbf{C}$ be a dagger category with dagger biproducts and let $K$ be an object of $\mathbf{C}$. The action $\cdot\colon \mathbf{C}(K,A)\times\mathbf{C}(K,K)\to \mathbf{C}(K,A)$ is a right scalar multiplication of the semiring $\langle \mathbf{C}(K,K),+, \circ \rangle $ on the commutative monoid $\langle \mathbf{C}(K,A),+\rangle$. The homset $\mathbf{C}(K,A)$ is then a right semimodule over $\mathbf{C}(K,K)$.
\end{lemma}

The following is a property of dagger biproducts that results from \cite[Lemma 2.12]{Selinger2007}: 

\begin{lemma}
Let $\mathbf{C}$ be a dagger category with dagger biproducts. Then the
dagger distributes over the biproducts in the sense that $(f\oplus g)^\dagger = f^\dagger\oplus g^\dagger$ for each $f,g\colon A\to B$. Also, $\Delta^\dagger =\nabla$ for the diagonal map $\Delta $ and codiagonal map $\nabla$.
\end{lemma}

\begin{cor}[{\cite[Corollary 2.42]{Heunen2019}}]\label{daggerdistribution}
Let $\mathbf{C}$ be a dagger category with dagger biproducts, then 
\begin{align*} 
(f+g)^\dagger = (\nabla\circ (f\oplus g) \circ \Delta )^\dagger =  \Delta^\dagger\circ (f\oplus g)^\dagger \circ \nabla^\dagger = \nabla\circ (f^\dagger \oplus g^\dagger) \circ \Delta = f^\dagger +g^\dagger.
\end{align*}
for each $f,g\in \mathbf{C}(A,B)$.
\end{cor}
The dagger biproduct on a dagger category $\mathbf{C}$ also allows for the notion of a sesquilinear map on $\mathbf{C}(K,H)$. 
\begin{prop}\label{sesquilinearform}
Let $\mathbf{C}$ be a dagger category with dagger biproducts. Then the right semimodule $\langle\mathbf{C}(K,A),+,\cdot,0\rangle$ with involution $\dagger$ has a left linear sesquilinear form. 
\begin{align*}
\langle\cdot,\cdot\rangle\colon \mathbf{C}(K,A)\times \mathbf{C}(K,A)&\to \mathbf{C}(K,K)\\
(f,g)&\mapsto \langle f,g\rangle:=g^\dagger f
\end{align*}
\end{prop}
\begin{proof}
Let $\lambda,\mu\in \mathbf{C}(K,K)$ and $f,g,h\in \mathbf{C}(K,A)$. Linearity in the first argument is a direct calculation, 

\begin{align*}
\langle f\cdot \lambda+g\cdot\mu, h \rangle 
= h^\dagger \circ (f\circ \lambda+g\circ\mu) 
= (h^\dagger \circ f\circ \lambda) +(h^\dagger \circ g\circ \mu)
= \langle f,h\rangle \lambda +\langle g, h \rangle\mu
\end{align*}
and conjugate symmetry follows from the contravariance of the dagger, as in
\[
\langle f, g\rangle = g^\dagger\circ f = (f^\dagger\circ g)^\dagger =
  \langle g,f\rangle ^\dagger. \qedhere
\]
\end{proof}
\section{Scalars}

\begin{Definition}[Dagger equalizer]
A \textit{dagger equalizer} is an equalizer which is also a dagger monomorphism.
\end{Definition}

\begin{prop}\label{factorisation}
Let $\mathbf{C}$ be a dagger category with dagger biproducts and dagger equalizers. Then any morphism of $\mathbf{C}$ is the composite of an epimorphism followed by a dagger monomorphism. 
\end{prop}
\begin{proof}
Let $f\colon A\to B$ be a morphism of $\mathbf{C}$. We then take
$m\colon X\to B$ to be the dagger equalizer of the cokernel pair
$f_1,f_2\colon B\rightrightarrows C$ of $f$ and $e\colon A\to X$ to be
the unique factorization of $f$ through $m$. It remains to show that $e$ is an epimorphism. 

This follows from (the dual of) a standard argument from the theory of regular categories, using the fact that dagger monomorphisms are split, and so any pushout of a dagger monomorphism is a (split) monomorphism. See \cite[Theorem 2.1.3]{Borceux1994V2}. 
\end{proof}

\begin{Definition}[Simple object]
An object $A$ in a dagger category with a zero object $O$ is \textit{simple}
when $O$ and $A$ are its only subobjects and $O\neq A$.
\end{Definition}

\begin{lemma}\label{multiplicativeinverse}
Let $\mathbf{C}$ be a dagger category with a simple object $K$, dagger biproducts and dagger equalizers. Then the semiring $\langle \mathbf{C}(K,K),+,\circ \rangle$ has multiplicative inverses.
\end{lemma}
The proof is essentially the argument for Schur's lemma.
\begin{proof}
We have already established that $(\mathbf{C}(K,K),+,\circ)$ is a semiring with involution $\lambda \mapsto \lambda^\dagger$. It remains to show that $\mathbf{C}(K,K)$ has multiplicative inverses. Let $\lambda\in \mathbf{C}(K,K)$ with $\lambda\neq 0$. Then $\lambda$ factors as $\lambda=m\circ e$ where $m$ is a dagger monomorphism and $e$ is an epimorphism by Lemma \ref{factorisation}. Since $K$ is simple, $m$ must be either 0 or an isomorphism, and since $m\circ e = \lambda\neq 0$ it follows that $m$ must be an isomorphism. Thus $\lambda$ is an epimorphism and $\lambda^\dagger$ is a monomorphism. Since $K$ is simple $\lambda^\dagger$ must be either 0 or an isomorphism, but $\lambda\neq 0$ and so $\lambda^\dagger$ is an isomorphism. Therefore $\lambda$ is an isomorphism.
\end{proof}

\begin{Definition}[Dagger kernel]
A \textit{dagger kernel} is a kernel which is also a dagger monomorphism.
\end{Definition}

The following proposition is a result found in \cite[Lemma 2.3]{Heunen2011}.

\begin{prop}\label{monoiffkerzero}
Let $\mathbf{C}$ be a dagger category with dagger equalizers and suppose that every dagger monomorphism is a kernel. Then $f$ is a monomorphism if and only if $\ker f = 0$.
\end{prop}

\begin{Definition}[Generators and dagger generators]
We say that $X$ is a {\em generator} if whenever $f,g\colon
  A\rightrightarrows B$ are distinct there exists a morphism $x\colon
  X\to A$ with $fx\neq gx$. We say that $X$ is a {\em dagger generator} if $x$ can be chosen to be a dagger monomorphism.
\end{Definition}

\begin{lemma}\label{divisionring}
Let $\mathbf{C}$ be a dagger category with a simple generator $K$, dagger biproducts, dagger equalizers, and the property that each dagger monomorphism is a kernel. Then the semiring $\langle \mathbf{C}(K,K),+,\cdot \rangle$ has additive inverses and so is a division ring.
\end{lemma}
This proof essentially follows the argument in \cite[Lemma 1]{Heunen2022}
\begin{proof}
Consider the kernel of $\nabla\colon K\oplus K\to K$. This has the
form $(x\oplus y)\Delta\colon A\to K\oplus K$ for some $x,y\colon A\to
K$. Then $x+y = \nabla(x\oplus y)\Delta = 0$. Either (i) $(x\oplus
y)\Delta  \neq 0$ or (ii) $(x\oplus y)\Delta = 0$. We will show that
additive inverses exist when (i) holds and that (ii) is impossible.
\begin{itemize}
\item[(i)]Suppose that $(x\oplus y)\Delta \neq0$ so that $x\neq 0$ or $y\neq 0$; without loss of generality, $x\neq 0$. Since $K$ is a generator, there exists a $z\colon K\to A$ such that $xz\neq 0$ and so the inverse $(xz)^{-1}$ exists as shown in Lemma \ref{multiplicativeinverse}. Now, 
\begin{align*}
1+(yz)\cdot(xz)^{-1} = (xz+yz)\cdot(xz)^{-1} = ((x+y)z)\cdot(xz)^{-1} = (0z)\cdot(xz)^{-1} = 0
\end{align*}
and so $1$ has an additive inverse. It follows that additive inverses exist in general.
\item[(ii)] If $\ker\nabla = 0$ then $\nabla$ is a monomorphism by
  Proposition \ref{monoiffkerzero}. The definition of $\nabla$ tells
  us that $\nabla i=\nabla j$, where $i$ and $j$ are embeddings for
  $K\oplus K$, but then $i=j$. It follows that $f=g$ for any two maps
  $f,g\colon K\to X$ and in particular $1=0$. This contradicts
  the simplicity of $K$ and so $\ker \nabla\neq 0$. \qedhere
\end{itemize}
\end{proof}
The involutive division ring structure on $\mathbf{C}(K,K)$ from Lemma \ref{divisionring} together with the semimodule structure of $\mathbf{C}(K,A)$ over $\mathbf{C}(K,K)$ from Lemma \ref{semimodule} gives the following result:

\begin{cor}\label{functorvect}
Sending $H\mapsto \mathbf{C}(K,H)$ and $f\mapsto \mathbf{C}(K,f)$ defines a functor, 
\begin{align*}
\mathbf{C}(K,-)\colon \mathbf{C}\to \mathbf{Vect}_{\mathbf{C}(K,K)}
\end{align*}
where $\mathbf{Vect}_{\mathbf{C}(K,K)}$ is the category of vector spaces and linear maps over the involutive division ring $\mathbf{C}(K,K)$.
\end{cor}

\section{Orthogonality}
For this section we assume that $\mathbf{C}$ is a dagger category with a simple generator $K$, dagger biproducts, dagger equalizers, and the property that each dagger monomorphism is a kernel.
Partway through the section, we shall introduce our final
  assumption, involving directed colimits of dagger monomorphisms.

\begin{Definition}[Hermitian Form]
Let $\mathbb{K}$ be a division ring with involution. For a right
$\mathbb{K}$-vector space $\mathcal{H}$, a \textit{Hermitian form} on
$\mathcal{H}$ is a sesquilinear map $\langle \cdot,\cdot \rangle\colon
\mathcal{H}\times \mathcal{H} \to \mathbb{K}$ which is conjugate
symmetric and non-singular: if $\langle u,v \rangle = 0$ for all
  $v\in V$ then $u=0$. A (right) $\mathbb{K}$-vector space with a Hermitian form is called a \textit{(right) $\mathbb{K}$-Hermitian space}. If $\mathcal{S}$ is a subset of a $\mathbb{K}$-Hermitian space $\mathcal{H}$, we can form the orthogonal complement
\begin{align*}\mathcal{S}^\perp :=\{x\in \mathcal{H}\mid \langle x,y\rangle=0, y\in \mathcal{S}\}.
\end{align*}
\end{Definition}

\begin{example}
The generator property of $K$ makes the sesquilinear form $\langle\cdot,\cdot\rangle\colon \mathbf{C}(K,A)\times \mathbf{C}(K,A)\to \mathbf{C}(K,K)$ non-singular and so $\mathbf{C}(K,A)$ is a Hermitian space.
\end{example}

\begin{Definition}[Closed subspace]
A subspace $\mathcal{F}$ of a Hermitian space $\mathcal{H}$ is said to be \textit{closed} when $(\mathcal{F}^\perp)^\perp = \mathcal{F}$.
We let $\mathrm{Sub}_C(\mathcal{H})$ denote the poset of all closed subspaces of a Hermitian space~$\mathcal{H}$. 
\end{Definition}

\begin{Definition}[Orthomodular]
A Hermitian space $\mathcal{H}$ said to be \textit{orthomodular} if for any closed subspace $\mathcal{F}\subseteq \mathcal{H}$ we have $\mathcal{H}=\mathcal{F}\oplus \mathcal{F}^\perp$. Note that there is always an injective map $\mathcal{F}\oplus \mathcal{F}^\perp \to \mathcal{H}$.
\end{Definition} 

\begin{Definition}[Dagger Subobject]
Let $A$ be an object of a dagger category $\mathbf{C}$. A \textit{subobject} of $A$ is an isomorphism class of monomorphisms with codomain $A$. A \textit{dagger subobject} of $A$ is a subobject of $A$ which contains a dagger monomorphism. 
\end{Definition}
\begin{lemma}\label{pullbackstability} 
Dagger monomorphisms are stable under pullback.
\end{lemma}
\begin{proof}
If $f\colon A\to H$ is a dagger monomorphism then it is the dagger equalizer of $1_H$ and $f f^\dagger$. Let $g\colon B\to H$ be any morphism and let $f'\colon P\to B$ be the dagger equalizer of $g$ and $f f^\dagger g$. Then there exists a unique morphism $g'\colon P\to M$ such that
\begin{equation*}
\xy*!C\xybox{\xymatrix @R+0pc @C+0pc{
P\ar@{.>}[r]^{g'}\ar[d]_{f'}&A\ar[d]^{f}&\\
B\ar[r]_{g}&H\ar@<+0.5ex>[r]^{f f^\dagger}\ar@<-0.5ex>[r]_{1_H}&H
}}\endxy
\end{equation*}
commutes, and now the square is a pullback.
\end{proof}

\begin{Definition}[Ortholattices]

An \textit{ortholattice} is a partially ordered set $(L,\leq)$ equipped with an involutive order isomorphism $\perp\colon  L^{\mathrm{op}}\to L\colon a\mapsto a^\perp$ satisfying the following conditions: 

\begin{itemize}
\item[(i)] The poset $L$ has finite meets (and so also finite joins);

\item[(ii)] The poset $L$ has a top element $1$ (and so also a bottom element $0$);

\item[(iii)] For each $a\in L$, $a\vee a^\perp = 1$ (and so also $a\wedge a^\perp = 0$).
\end{itemize}
\end{Definition}

\begin{example}
Given an object $H$ in $\mathbf{C}$, the poset $\langle
\mathrm{Sub}_C(\mathbf{C}(K,H)),\subseteq \rangle$ is an ortholattice
where each $\mathcal{F}\in \mathrm{Sub}_C(\mathbf{C}(K,H))$ has orthocomplement $\mathcal{F}^\perp :=\{ f\in \mathbf{C}(K,H)\mid g^\dagger f = 0,\, g\in \mathcal{F} \}$.
\end{example}

A dagger subobject of an object $H$ has a representative dagger monomorphism. Denote the collection of all dagger subobjects by $\mathrm{Sub}_\dagger(H)$. The usual ordering in $\mathrm{Sub}(H)$ restricts to $\mathrm{Sub}_\dagger(H)$: for dagger subobjects $f\colon A\to H$ and $g\colon B\to H$ we have $f\leq g$ if and only if there exists a morphism $h\colon A\to B$ (necessarily a dagger monomorphism) such that $f=g h$.
 \\
 \\
 The following Lemma is based on results found in \cite{Heunen2011}.

\begin{lemma}
$\langle \mathrm{Sub}_\dagger(H),\leq \rangle$ is an ortholattice with $f^\perp = \ker(f^\dagger)\colon A^\perp \to H$ for each $f\in \mathrm{Sub}_\dagger{H}$.
\end{lemma}

\begin{proof} 
We first show that $(-)^\perp\colon \mathrm{Sub}_\dagger(H)\to
\mathrm{Sub}_\dagger(H)^{op}\colon f\mapsto f^\perp$ is an involutive
order isomorphism. Suppose $f\leq g$, so that there exists an $h$ such that $f=g h$. Since $f^\dagger g^\perp = (h^\dagger g^\dagger)  \ker(g^\dagger) = 0$, there exists a unique $k$ for which $g^\perp = \ker(f^\dagger)  k = f^\perp  k$ and hence $g^\perp\leq f^\perp$. Let $f\in \mathrm{Sub}_\dagger(H)$, then $(f^\perp)^\perp = \ker(\ker(f^\dagger)^\dagger) = \ker(\cok(f))= f$, since all dagger monomorphisms are (dagger) kernels. 

The bottom of $\mathrm{Sub}_\dagger(H)$ is the zero morphism $0\colon
O \to H$ and the top is the identity $1\colon H\to H$. The meet of $f,g\in \mathrm{Sub}_\dagger(H)$ is the pullback map $f\wedge g$ defined as
\begin{equation*}
\xy*!C\xybox{\xymatrix @R+0pc @C+0pc{
A\wedge B \ar[r]\ar[d]\ar@{.>}[dr]^{f\wedge g}&A\ar[d]^{f}\\
B\ar[r]_{g}&H
}}\endxy
\end{equation*}
which we we can take to be a dagger subobject by Lemma~\ref{pullbackstability}. Finally, consider $f\wedge f^\perp$: this is the pullback
\begin{equation*}
\xy*!C\xybox{\xymatrix @R+0pc @C+0pc{
A\wedge A^\perp\ar[r]^{p}\ar[d]_{q}\ar@{.>}[dr]^{f\wedge f^\perp}&A\ar[d]^{f}\\
A^\perp\ar[r]_{f^\perp}&H.
}}\endxy
\end{equation*}
Since $fp = f^\perp q$ we have $p = f^\dagger f p= f^\dagger f^\perp q
= f^\dagger \ker(f^\dagger ) q = 0$. Therefore $f\wedge f^\perp = 0 $ and dually $f\vee f^\perp  = 1$.
\end{proof}

If we want $\mathrm{Sub}_{\dagger}(H)$ to be complete as a lattice we need additional structure.

\begin{Definition}[Directed Colimit]
Let $\mathbf{C}$ be a category. A \textit{directed colimit} is a colimit of a functor $D\colon \mathscr{J}\to \mathbf{C}$ where $\mathscr{J}$ is a directed poset.
\end{Definition}

We now also assume the wide subcategory $\mathbf{C}_\mathrm{dm}$ of dagger monomorpisms of $\mathbf{C}$ has directed colimits.

\begin{lemma}
The ortholattice $\mathrm{Sub}_\dagger(H)$ is complete for each $H\in\mathbf{C}$.
\end{lemma}

\begin{proof}
We know that $\mathrm{Sub}_{\dagger}(H)$ has finite joins, so it
will suffice to show that it has directed joins. But $\mathbf{C}_\mathrm{dm}$ has
directed joins and so the slice category $\mathbf{C}_\mathrm{dm}/H$
has directed joins, formed as in $\mathbf{C}_\mathrm{dm}$, and
$\mathbf{C}_\mathrm{dm}/H$ is equivalent to
$\mathrm{Sub}_{\dagger}(H)$.
\end{proof}

\begin{theorem}\label{subtoclosed}
Let $H$ be an object in $\mathbf{C}$. The function
\begin{align*}
\varphi\colon  \mathrm{Sub}_\dagger(H)&\to \mathrm{Sub}_C(\mathbf{C}(K,H))\\
( f\colon A\to H)&\mapsto \varphi(f) := \{ fa\mid a\in \mathbf{C}(K,A)\}
\end{align*}
is an isomorphism of ortholattices. 
\end{theorem}
\begin{proof}
We first show that $\varphi$ preserves orthocomplements:
\begin{align*}
x\in\varphi(f)^\perp &\iff \forall a\in \mathbf{C}(K,A): \langle fa,x \rangle = x^\dagger fa= 0 \\
&\iff x^\dagger f = 0 &(\text{generator property of $K$})\\
&\iff f^\dagger x = 0\\
&\iff \exists! b : x = \ker(f^\dagger) h = f^\perp b\\
&\iff x\in \varphi(f^\perp)
\end{align*}
hence $\varphi(f)^\perp = \varphi(f^\perp)$. It follows that $\varphi(f)^{\perp\perp} =\varphi(f^\perp)^\perp = \varphi(f^{\perp\perp}) = \varphi(f)$ and so $\varphi(f)$ is closed. 

To see that $\varphi$ preserves order, let $f\colon A\to H$ and $g\colon B\to H$ be dagger subobjects of $H$ with $f\leq g$.  Any element of $\varphi(f)$ is of the form $fa$ for $a\in \mathbf{C}(K,A)$. Since $f=gh$ for some $h\colon A\to B$, we have $ha\in \mathbf{C}(K,B)$ and so $fa=gha\in \varphi(g)$, hence $\varphi(f)\subseteq\varphi(g)$. To see that $\varphi$ reflects order, let $\varphi(f)\leq\varphi(g)$. For any $x\colon K\to H$ we have $ff^\dagger x\in \varphi(f)$, and so $ff^\dagger x\in \varphi(g)$; thus $gg^\dagger ff^\dagger x = ff^\dagger x$. 
The generator property of $K$ means that $ff^\dagger  =gg^\dagger
ff^\dagger $ and so $f=gg^\dagger f$, hence $f\leq g$. Thus $\varphi$ reflects order and is also injective.

To see that $\varphi$ is surjective, suppose that
$\mathcal{F}\subseteq \mathbf{C}(K,H)$ with $\mathcal{F}^{\perp\perp}
= \mathcal{F}$ and let $m$ be the join of all $m_i\colon M_i\to
\mathcal{F}$ with $\varphi(m_i)\subseteq \mathcal{F}$: this is where we
  use completeness of $\mathrm{Sub}_{\dagger}(H)$. We show that $\mathcal{F}\subseteq \varphi(m)\subseteq \mathcal{F}^{\perp\perp}$.
\begin{itemize}
\item[(i)] Let $f\in \mathcal{F}$. By Proposition \ref{factorisation}, $f$ decomposes as,
\begin{equation*}
\xy*!C\xybox{\xymatrix @R+0pc @C+2pc{
K\ar[r]^{e} & X \ar[r]^{n}&H
}}\endxy 
\end{equation*} 
where $n$ is a dagger monomorphism and $e$ is an epimorphism. Then let $y\in \mathcal{F}^\perp$ so then $ e^\dagger n^\dagger y = f^\dagger y =0$, and hence $ n^\dagger y =0$. It follows for each $x\colon K\to X$ that $ (nx)^\dagger y=x^\dagger n^\dagger y =0$ and so $nx\in \mathcal{F}^{\perp\perp} = \mathcal{F}$ for each $x$. It follows that $\varphi(n)\subseteq \mathcal{F}$ and hence $n\leq m$ by definition of $m$. Therefore $f=ne\in \varphi(m)$ and $\mathcal{F}\subseteq \varphi(m)$.

\item[(ii)]
First we show that if  $y\in \mathcal{F}^\perp$ then
    $y^\dagger m=0$. Factorize $y$ as an epimorphism $e$ followed by a
    dagger monomorphism $n$.  Since $m_i x\in \mathcal{F}$
  for each $i$ and each $x\colon K\to M_i$ we have $y^\dagger m_i x = 0$ and so $y^\dagger m_i =e^\dagger n^\dagger
  m_i=0$ by the generator property of $K$, and now $n^{\dagger}m_i=0$ since $e$ is an epimorphism. It
  follows that $m_i= \ker(n^\dagger) h$ for some $h$ and hence
  $m_i\leq n^\perp$ for all $i$. Since $m$ is the join of the $m_i$ it follows
  that $m\leq n^\perp$ and hence $m\perp n$; thus $y^\dagger m =
  e^\dagger n^\dagger m = 0$.

Thus for any $y\in \mathcal{F}^\perp$ and $g\colon K\to M$ we have
$y^{\dagger}mg=0$, and so $mg\in\mathcal{F}^{\perp\perp}$ and
$\varphi(m)\subseteq\mathcal{F}^{\perp\perp}$. 
\qedhere
\end{itemize}
\end{proof}
\begin{lemma}
A dagger subobject $m\colon M\to H$ and its orthocomplement $m^\perp\colon M^\perp\to H$ define a dagger biproduct $H\cong M\oplus M^\perp$ with
\begin{equation*}
\xy*!C\xybox{\xymatrix @R+1pc @C+2pc{
M\ar@<+0.5ex>[r]^{m}&\ar@<+0.5ex>[l]^{m^\dagger}H\ar@<-0.5ex>[r]_{m^{\perp\dagger}}&M^\perp. \ar@<-0.5ex>[l]_{m^{\perp}} 
}}\endxy 
\end{equation*} 
\end{lemma}
\begin{proof}
This follows immediately from the fact that $m^\dagger m=1$ and $m^\bot=\ker(m^\dagger)$.
\end{proof}

\begin{theorem}\label{theoremorthomodularspace}
$\mathbf{C}(K,H)$ is an orthomodular space.
\end{theorem}

\begin{proof}
Let $F$ be a closed subspace of $\mathbf{C}(K,H)$. By Theorem \ref{subtoclosed}, there exists a dagger subobject $m\colon M\to H$ such that $\varphi(m)= F$. Observe that for any $h\in \mathbf{C}(K,H)$, $mm^\dagger h\in \varphi(m)$ and $m^\perp m^{\perp\dagger} h\in \varphi(m^\perp)=\varphi(m)^\perp$. It follows from the previous lemma that $1_H= mm^\dagger+m^\perp m^{\perp\dagger}$ and so 
$
h= 1_H h = (mm^\dagger+m^\perp m^{\perp\dagger})h = mm^\dagger h+m^\perp m^{\perp\dagger}h
$. 
Thus $\mathbf{C}(K,H)\cong \varphi(m)\oplus \varphi(m)^\perp$.
\end{proof}

\section{Orthonormal Bases}

We have now seen all of the ingredients in the definition
  of Hilbert category given in the introduction, and for the remainder
  we shall suppose that they hold. This is also a good moment to
  observe that these conditions hold in $\mathbf{Hilb}_\mathbb{K}$.

\begin{example}
The category $\mathbf{Hilb}_{\mathbb{K}}$ is a Hilbert category
  if $\mathbb{K}$ is $\mathbb{R}$, $\mathbb{C}$, or $\mathbb{H}$.
The existence of adjoints for any $f\in \mathbf{Hilb}_\mathbb{K}$ implies (D). The 1-dimensional space $\mathbb{K}$ is a simple generator as in (G), while (B) holds because of the
  0-dimensional space and the usual direct sum of Hilbert spaces. For
bounded linear maps $f,g\colon \mathcal{H}\to\mathcal{K}$ the
embedding of the closed subspace $\{x\in \mathcal{H}\mid
fx=gx\}\subseteq\mathcal{H}$ is an equalizer and so (E) is
satisfied. A dagger monomorphism $f\colon
\mathcal{H}\hookrightarrow\mathcal{K}$ is a kernel of the orthogonal projection $g\colon \mathcal{K}\cong \mathcal{H}\oplus \mathcal{H}^\perp \to \mathcal{H}^\perp$, giving (K). For (C), any directed union of closed subspaces forms a pre-Hilbert space and taking its closure results in a Hilbert space which defines the directed colimit.
\end{example}

To discuss a notion of orthonormal basis for an object $H$ of a Hilbert category $\mathbf{C}$ we first define an object $\ell^2X$ for a set $X$, generalising the usual notion when $\mathbf{C}=\mathbf{Hilb}$. Let $D\colon  \mathrm{Sub}_{\mathrm{fin}}X \to \mathbf{C}_{\mathrm{dm}}$ be the functor from the directed poset $\mathrm{Sub}_\mathrm{fin}X$ of finite subobjects of $X$ to $\mathbf{C}_{\mathrm{dm}}$ which takes a finite subset $A\subseteq X$ to the dagger biproduct $\bigoplus_{a\in A} K$. We can then form the directed colimit $\ell^2 X$ over $D$ with components 
\begin{align*}
\{x_A\colon \bigoplus_{a\in A} K \to \ell^2 X\}_{A\in \mathrm{Sub}_{\mathrm{fin}}X}.
\end{align*}

\begin{lemma}\label{orthonormalsequence}
The family $(x_a\colon  K \to \ell^2X)_{a\in X}$ is an orthonormal
  sequence in $\mathbf{C}(K,\ell^2X)$. Composition with the
    $x_a$ induces a bijection between the set of dagger monomorphisms
    $\ell^2X\to H$ and the set of orthonormal families in
    $\mathbf{C}(K,H)$ indexed by $X$.
\end{lemma}
\begin{proof}
  Given a finite family $(m_a\colon K\to H)_{a\in A}$ of morphisms
  there is a unique induced $m_A\colon \oplus_{a\in A}K\to H$, and
  $m_A$ is a dagger monomorphism if and only if the family
  $(m_a)_{a\in A}$ is orthonormal. An infinite family $(m_a\colon K\to
  H)_{a\in X}$ is orthonormal if and only if each finite subfamily is
  orthonormal. Given such a family we obtain a dagger monormophism
  $m_A\colon \oplus_{a\in A}K\to H$, and these form a cocone as $A$
  varies, and thus a unique induced dagger monomorphism $m_X\colon\ell^2X\to H$.
\end{proof}

\begin{Definition}[Orthonormal basis]\label{deforthonormalbasis}
Let $\mathbf{C}$ be a Hilbert category. Given an $H\in \mathbf{C}$, an
orthonormal family $(m_x)_{x\in X}$ of morphisms $K\to H$ is
called an \textit{orthonormal basis} for $H$ when the induced
  dagger monomorphism $m_X\colon  \ell^2 X\to H $ is invertible. 
\end{Definition}
Definition \ref{deforthonormalbasis} matches the usual definition of an orthonormal basis for a Hilbert space since when given an orthonormal basis $X$ for a Hilbert space $\mathcal{H}$ we have $\mathcal{H}\cong \ell^2X\cong \overbar{\mathrm{span}X}$. 

For the following proposition we replace the condition (G) from Definition \ref{defhilbertcategory} with the following:
\begin{itemize}
\item[($\mathrm{G}^\dagger$)] $\mathbf{C}$ is equipped with a simple dagger generator $K$.
\end{itemize}
Since a one-dimensional Hilbert space is clearly a simple dagger
generator, this apparently stronger condition holds in
$\mathbf{Hilb}$. We shall see in Section~6 that any Hilbert category is
equivalent to $\mathbf{Hilb}$, and so in the presence of the other
axioms, ($\mathrm{G}^{\dagger}$) is actually not stronger than (G).

\begin{prop}\label{orthonormalbase}
 In a Hilbert category satisfying $\mathrm{(G^{\dagger})}$, any
 object $H$ has an orthonormal basis.
\end{prop}

\begin{proof}
Zorn's lemma tells us that the poset $\langle\mathcal{O},\subseteq\rangle$ of all orthonormal sets in $\mathbf{C}(K,H)$ contains a maximal element since the union of a chain in $\mathcal{O}$ forms an orthonormal set and so each chain in $\mathcal{O}$ has an upper bound. Call this maximal element $X$. We now show that $X$ is in fact a basis for $H$. The universal property of $\ell^2 X$ tells us that there exists a unique dagger monomorphism $m\colon \ell^2X\to H$ determined by the orthonormal family $X\subset \mathbf{C}(K,H)$; we are to show that $m$ is invertible. For $m^\perp := \ker(m^\dagger)\colon \ell^2X^\perp\to H$ we know that,
\begin{equation*}
\xy*!C\xybox{\xymatrix @R+1pc @C+2pc{
\ell^2 X\ar@<+0.5ex>[r]^{m}&\ar@<+0.5ex>[l]^{m^\dagger}H\ar@<-0.5ex>[r]_{m^{\perp\dagger}}&\ell^2 X^\perp\ar@<-0.5ex>[l]_{m^{\perp}}
}}\endxy 
\end{equation*} 
forms a biproduct with $1_H = mm^\dagger+m^\perp m^{\perp\dagger}$ and so $H\cong \ell^2 X\oplus \ell^2 X^\perp$. For the sake of contradiction assume $ \ker(m^\dagger) \neq 0$; then there exists a dagger monomorphism $f\colon K\to H$ with $m^\dagger f =0$ since $K$ is a dagger generator. Now for each non-zero $a\in X$, there exists an $x_a\colon  K\to \ell^2X $ for which $a=mx_a$.
\begin{equation*}
\xy*!C\xybox{\xymatrix @R+1pc @C+2pc{
\ell^2X \ar[r]^{m}&H&\ell^2X ^\perp\ar[l]_{\ker(m^\dagger)}\\
&K\ar@<-0.5ex>[u]_{f}\ar@<+0.5ex>[u]^{a}\ar[ul]^{x_a}&
}}\endxy 
\end{equation*}
It follows that $\langle f,a \rangle = a^\dagger f = (mx_a)^\dagger f = x_a^\dagger m^\dagger f = x^\dagger_a 0= 0$ and since $a$ is arbitrary $f$ is orthogonal to each $a\in X$. By taking the union $\{f\}\cup X$ we have an orthonormal set which strictly contains $X$, but this contradicts the maximality of $X$. Thus $\ker(m^\dagger) = 0$ and so $H\cong \ell^2X$ and hence $X$ is an orthonormal basis for $H$.
\end{proof}
Thus we recover the following well-known fact:
\begin{cor}\label{hilbbasis}
Any Hilbert space $\mathcal{H}$ has an orthonormal basis $X$. 
\end{cor}

\section{Equivalence}

\begin{prop}
Let $H$ be an object of a Hilbert category $\mathbf{C}$. The division
ring $\mathbf{C}(K,K)$ is isomorphic to $\mathbb{R}$, $\mathbb{C}$, or $\mathbb{H}$ and $\mathbf{C}(K,H)$ is a right $\mathbf{C}(K,K)$-Hilbert space.
\end{prop}
\begin{proof}
We know $\mathbf{C}(K,\ell^2\mathbb{N})$ to be an orthomodular space
by Theorem \ref{theoremorthomodularspace} and contain an infinite
orthonormal sequence $\{x_n\colon K\to\ell^2\mathbb{N\}}_{n\in
  \mathbb{N}}$ by Lemma \ref{orthonormalsequence}. Thus the
conditions for  Sol\'er's Theorem \ref{solerstheorem} hold and
so $\mathbf{C}(K,K)\in\{\mathbb{R,C,H}\}$ and
$\mathbf{C}(K,\ell^2\mathbb{N})$ is a Hilbert space.
Consider the orthomodular space $\mathbf{C}(K,\ell^2\mathbb{N}\oplus
H)$. The dagger monomorphism $\mathbf{C}(K,\ell^2\mathbb{N})\to
  \mathbf{C}(K,\ell^2\mathbb{N}\oplus H)$ determines an infinite orthonormal
  sequence by Lemma~\ref{orthonormalsequence}, and so
  $\mathbf{C}(K,\ell^2\mathbb{N}\oplus H)$ is also a Hilbert
  space.

Finally $\mathbf{C}(K,\ell^2\mathbb{N})$ is a closed subspace of the Hilbert space $\mathbf{C}(K,\ell^2\mathbb{N}\oplus H)$ and so its orthogonal complement $\mathbf{C}(K,H)$ is also a Hilbert space. 
\end{proof}

\begin{lemma}\label{functorhilb}
The functor $\mathbf{C}(K,-)\colon \mathbf{C}\to \mathbf{Vect}_{\mathbf{C}(K,K)}$ lifts to a dagger functor
\begin{align*}
\mathbf{C}(K,-)\colon \mathbf{C}\to \mathbf{Hilb}_{\mathbf{C}(K,K)}.
\end{align*}
\end{lemma}
\begin{proof}
In Corollary \ref{functorvect} we saw that $\mathbf{C}(K,-)\colon \mathbf{C}\to \mathbf{Vect}_\mathbf{C}(K,K)$ is a functor. For each $H\in \mathbf{C}$, $\mathbf{C}(K,H)$ is a Hilbert space and so to lift the codomain of $\mathbf{C}(K,-)$ to $\mathbf{Hilb}$ we require $\mathbf{C}(K,f)$ to be bounded and $\mathbf{C}(K,f^\dagger)=\mathbf{C}(K,f)^\dagger$ for each $f\colon H_1\to H_2$ in $\mathbf{C}$. Let $f\colon H_1\to H_2$ be a morphism in $\mathbf{C}$ and then for each $h_1\in \mathbf{C}(K,H_1)$ and $h_2 \in \mathbf{C}(K,H_2)$,
\begin{align*}
\langle \mathbf{C}(K,f^\dagger)(h_2), h_1\rangle 
&=  \langle f^\dagger \circ h_2, h_1\rangle\\
&=  h_1^\dagger\circ f^\dagger\circ h_2\\
&= (f\circ h_1)^\dagger\circ h_2\\
&= \langle h_2,f\circ h_1 \rangle\\
&=  \langle h_2, \mathbf{C}(K,f)(h_1)\rangle 
\end{align*}
and so $\mathbf{C}(K,f^\dagger)$ is the adjoint to $\mathbf{C}(K,f)$; that is, $\mathbf{C}(K,f^\dagger) = \mathbf{C}(K,f)^\dagger$. It follows from Theorem \ref{adjointtoboundedness} that $\mathbf{C}(K,f)$ is bounded for each $f\in \mathbf{C}$.
\end{proof}

\begin{Definition}[Dagger Equivalence]
A \textit{dagger equivalence} between dagger categories $(\mathbf{C},\dagger)$ and $(\mathbf{D},\dagger)$ is a dagger functor $F\colon \mathbf{C}\to \mathbf{D}$ which is full, faithful,  and surjective on objects up to unitary isomorphism.
\end{Definition}

The rest of this section will be devoted to proving that $\mathbf{C}(K,-)\colon \mathbf{C}\to \mathbf{Hilb}_{\mathbf{C}(K,K)}$ is a dagger equivalence. We begin with faithfulness.

\begin{lemma}\label{faithful}
The dagger functor $\mathbf{C}(K,-)\colon \mathbf{C}\to \mathbf{Hilb}_{\mathbf{C}(K,K)}$ is faithful.
\end{lemma} 
\begin{proof}
Let $f,g\colon H\to K$ be morphisms of $\mathbf{C}$ and suppose $\mathbf{C}(K,f)=\mathbf{C}(K,g)$. Then for each $h\in \mathbf{C}(K,H)$ we have $f\circ h= g\circ h$, so by the generator property of $K$ we have $f=g$.
\end{proof}

\begin{lemma}\label{essentiallysurjective}
The dagger functor $\mathbf{C}(K,-)\colon \mathbf{C}\to \mathbf{Hilb}_{\mathbf{C}(K,K)}$ is essentially surjective on objects.
\end{lemma}

\begin{proof}
Let $\mathcal{H}\in \mathbf{Hilb}_{\mathbf{C}(K,K)}$ and let $X$ be an orthonormal basis for $\mathcal{H}$. We can construct $\ell^2_{\mathbf{C}}X$ in  $\mathbf{C}$ with orthonormal basis $X$. The canonical subset of components $\{x_a\colon K\to \ell^2_\mathbf{C} X\}_{a\in X}$ forms an orthonormal basis for the space $\mathbf{C}(K,\ell^2_\mathbf{C}X)$. Since both $\mathcal{H}$ and $\mathbf{C}(K,\ell^2_\mathbf{C}X)$ share bases of the same cardinality we have $\mathcal{H}\cong\mathbf{C}(K,\ell^2_\mathbf{C} X)$ in $\mathbf{Hilb}_{\mathbf{C}(K,K)}$.
\end{proof}
We now turn to fullness.
\begin{lemma}\label{finitedim}
Let $\mathbf{C}(K,H_1)$ be finite dimensional. Then for each bounded linear map $T\colon\mathbf{C}(K,H_1)\to \mathbf{C}(K,H_2)$, there exists a morphism $t\colon H_1\to H_2$ such that $T=\mathbf{C}(K,t)$.
\end{lemma}
\begin{proof}
If $\{x_a\colon K\to H_1\}_{a\in X}$ is an orthonormal basis for $\mathbf{C}(K,H_1)$, then the $x_a$ exhibit $H_1$ as the biproduct $\oplus_X K$. 
It follows that the $T(x_a)\colon K\to H_2$ induce a unique $t\colon H_1\to H_2$ with $t x_a=T( x_a)$ for each $a$. 
\begin{equation*}
\xy*!C\xybox{\xymatrix @R+1pc @C+2pc{
&H_2\\
K\ar[r]_{x_a}\ar[ru]^{T(x_a)}&H_1\ar@{.>}[u]_{t}
}}\endxy 
\end{equation*} 
Since $T$ is (bounded and) linear, it is defined by how it acts on each basis element $x_a$ but $ T(x_a) = tx_a =\mathbf{C}(K,t)(x_a)$ for each $a\in X$ and so $T=\mathbf{C}(K,t)$.
\end{proof}

\begin{lemma}\label{full}
The dagger functor $\mathbf{C}(K,-)\colon \mathbf{C}\to \mathbf{Hilb}_{\mathbf{C}(K,K)}$ is full.
\end{lemma}
To prove that $\mathbf{C}(K,-)$ is full, we need to show that each
$T\colon\mathbf{C}(K,H_1)\to \mathbf{C}(K,H_2)$ in
$\mathbf{Hilb}_{\mathbf{C}(K,K)}$ has a corresponding $t\colon H_1\to
H_2$ in $\mathbf{C}$. The case where $\mathbf{C}(K,H_1)$ is
  finite dimensional was shown in Lemma~\ref{finitedim}. The approach
to proving the infinite-dimensional case is as follows:
\begin{itemize}
\item[(i)] Reduce to the case where $\dim\mathbf{C}(K,H_1)\leq \dim\mathbf{C}(K,H_2)$;
\item[(ii)] Reduce to the case where $H_1=H_2$;
\item[(iii)] Reduce to the case where $T\colon \mathbf{C}(K,H)\to \mathbf{C}(K,H)$ is unitary (a dagger isomorphism);
\item [(iv)] Prove that for each unitary $U\colon \mathbf{C}(K,H)\to \mathbf{C}(K,H)$ there exists a $t\colon H\to H$ such that $U=\mathbf{C}(K,t)$.
\end{itemize} 

\begin{proof}

\begin{itemize}

\item[(i)] If $\dim\mathbf{C}(K,H_1)> \dim\mathbf{C}(K,H_2) $ then work with $T^\dagger\colon \mathbf{C}(K,H_2)\to \mathbf{C}(K,H_1)$ instead.

\item[(ii)] Suppose for each bounded linear map $T\colon \mathbf{C}(K,H_2)\to \mathbf{C}(K,H_2)$ there exists a $t\colon H_2\to H_2$ such that $T= \mathbf{C}(K,t)$. Then if $\widetilde{T}\colon \mathbf{C}(K,H_1)\to \mathbf{C}(K,H_2)$ is a bounded linear map with $\dim \mathbf{C}(K,H_1)\leq \dim\mathbf{C}(K,H_2)$, there exist orthonormal bases $X_1$ and $X_2$ for $\mathbf{C}(K,H_1)$ and $\mathbf{C}(K,H_2)$ respectively. It follows that $|X_1|\leq |X_2|$ and so there exists a dagger monomorphism $\ell^2X_1\hookrightarrow \ell^2 X_2$, and by Definition \ref{deforthonormalbasis}, a dagger monomorphism $m\colon H_1\to H_2$. It follows that there exists a $\tilde{t}\colon H_2\to H_2$ such that $\widetilde{T}\circ \mathbf{C}(K,m^\dagger)= \mathbf{C}(K,\tilde{t})$ and so,
\begin{equation*}
\xy*!C\xybox{\xymatrix @R+2pc @C+2pc{
&\mathbf{C}(K,H_1)\ar[r]^{\widetilde{T}}&\mathbf{C}(K,H_2)\\
\mathbf{C}(K,H_1)\ar[r]_{\mathbf{C}(K,m)}\ar[ru]^{\mathbf{C}(K,1_{H_1})}&\mathbf{C}(K,H_2)\ar[ru]_{\mathbf{C}(K,\tilde{t})}\ar[u]_{\mathbf{C}(K,m^\dagger)}&
}}\endxy 
\end{equation*} 
commutes, and $\widetilde{T}= \mathbf{C}(K,\tilde{t}m)$.

\item[(iii):] Let $T\colon \mathbf{C}(K,H)\to \mathbf{C}(K,H)$ be a bounded linear map. Then by Appendix \ref{appendixB} there exists a family of unitary maps $\{ U_1,\dots, U_N \}$ and a family of $\mathbf{C}(K,K)$-coefficients $\{\alpha_1\dots\alpha_N\}$ such that $T=\alpha_1 U_1+\dots+\alpha_N U_N$. Let $U_i = \mathbf{C}(K,t_i)$ for some $t_i\colon H\to K$ for each $i\in \{1,\dots, N\}$. Then for $t =\alpha_1 t_1+\dots+ \alpha_N t_N$,
\begin{align*}
T &= \alpha_1 U_1+\dots+\alpha_N U_N= \alpha_1  \mathbf{C}(K,t_1)+\dots+\alpha_N  \mathbf{C}(K,t_N)=\mathbf{C}(K,t).
\end{align*}

\item[(iv)] We may take $H$ to be $\ell^2 X$ where $X$ is an orthonormal basis for $\mathbf{C}(K,H)$. Let $U\colon \mathbf{C}(K,\ell^2X)\to \mathbf{C}(K,\ell^2 X)$ be unitary.
  
Since $\{x_{a}\}_{a\in X}$ is an orthonormal basis for $\mathbf{C}(K,\ell^2 X)$ and since $U$ is unitary, the collection of $U(x_{a})$ for each $a\in X$ forms an orthonormal basis in $\mathbf{C}(K,\ell^2 X)$.
Thus $\{U(x_{a})\}_{a\in X}$ is a cocone in $\mathbf{C}_\mathrm{dm}$. The universal property of $\ell^2 X$ induces a unique dagger monomorphism $t\colon \ell^2 X\to \ell^2 X$ such that,
\begin{equation*}
\xy*!C\xybox{\xymatrix @R+0pc @C+0pc{
&\ell^2 X\\
K\ar[r]_{x_{a}}\ar[ru]^{U(x_{a})}&\ell^2 X\ar@{.>}[u]_{t}
}}\endxy 
\end{equation*} 
commutes for each $a\in X$, hence $U(x_{a}) = \mathbf{C}(K,t)(x_{a})$,
and therefore $U=\mathbf{C}(K,t)$. \qedhere
\end{itemize}
\end{proof}

\begin{theorem}\label{daggerequivalence}
The dagger functor $\mathbf{C}(K,-)\colon \mathbf{C}\to \mathbf{Hilb}_{\mathbf{C}(K,K)}$ is a dagger equivalence.
\end{theorem}
\begin{proof}
Lemmas \ref{faithful},\ref{essentiallysurjective},\ref{full}.
\end{proof}

\section{Monoidal Structure}

\begin{Definition}[Dagger Monoidal Category]
A \textit{dagger monoidal category} is a dagger category
$(\mathbf{C},\dagger)$ with a monoidal structure
$(\otimes,I,\alpha,l,r)$ such that 
\begin{itemize}
\item the tensor product $\otimes\colon \mathbf{C}\times \mathbf{C}\to
  \mathbf{C}$ is a dagger functor, so that $(f\otimes g)^\dagger
    = f^\dagger\otimes g^\dagger$ for all $f,g$;
\item the associator $\alpha$, left unitor $l$ and right unitor $r$
  are unitary, which is to say that $\alpha_{A,B,C}$, $l_A$ and $r_A$ are unitary isomorphisms for each $A,B,C\in \mathrm{ob}\mathbf{C}$.
\end{itemize}
\end{Definition}

\begin{remark}
Equivalently, a category $\mathbf{C}$ is a dagger monoidal category if the contravariant functor $\dagger\colon \mathbf{C}^{\mathrm{op}}\to \mathbf{C}$ is a strict monoidal functor \cite[Definition 2.4]{Selinger2007}. In that case, the fact that the associator and unitors are unitary follows automatically. 
\end{remark}

If a Hilbert category $\mathbf{C}$ is equipped with monoidal structure as above then a scalar multiplication can be defined as: 
$
\bullet\colon  \mathbf{C}(A,B)\times\mathbf{C}(I,I)\to \mathbf{C}(I,A)\colon (f,\lambda)\mapsto f\bullet \lambda
$ 
where $f\bullet \lambda = r^{-1} (f\otimes \lambda) r$. Previously we defined a scalar multiplication 
$
\circ\colon \mathbf{C}(K,A)\times\mathbf{C}(K,K)\to \mathbf{C}(K,A)\colon 
(f,\lambda)\mapsto f\circ \lambda
$ 
where $\circ$ is morphism composition in $\mathbf{C}$. The following result relates the two.

\begin{prop}\label{uniquescalarmulti}
Let $\mathbf{C}$ be a Hilbert category equipped with a dagger monoidal structure with unit $K$. Then $\bullet=\circ$.
\end{prop}
\begin{proof}
Let $f\in \mathbf{C}(K,A)$ and $\lambda\in \mathbf{C}(K,K)$. Then $f\circ \lambda = f\bullet \lambda$ since
\begin{equation*}
\xy*!C\xybox{\xymatrix @R+0pc @C+0pc{
&K\otimes\ar[rr]^{f\otimes \lambda} K\ar[rd]_{\id\otimes \lambda}&&A\otimes K\ar[rd]^{r^{-1}_K}&\\
K\ar[rd]_{\lambda}\ar[ru]^{r_K}&&K\otimes K\ar[ru]_{f\otimes \id}\ar[rd]_{r^{-1}_K}&&A\\
&K\ar[ru]_{r_K}\ar[rr]_{\id}&&K\ar[ru]_{f}&
}}\endxy 
\end{equation*}
commutes.
\end{proof}

In \cite{Heunen2022}, Heunen and Kornell consider a category satisfying
the axioms in Definition \ref{defhilbertcategory} with (G) replaced by:
\begin{itemize}
\item[(T)]  $\mathbf{C}$ is equipped with a dagger monoidal structure $\otimes$ whose unit $K$ is a simple monoidal generator\footnote{
A \textit{monoidal generator} in a monoidal category is an object $G$ with the property that when given a pair $f,g\colon A\otimes B\rightrightarrows C$, if $f\circ (h\otimes k) = g\circ (h\otimes k)$ for each $h\colon G\to A$ and each $k\colon G\to B$ then $f=g$.}.
\end{itemize}
They prove:
\begin{theorem}[Theorem 10, \cite{Heunen2022}]\label{hktheorem}
The dagger functor $\mathbf{C}(K,-)\colon \mathbf{C}\to \mathbf{Hilb}_{\mathbf{C}(K,K)}$ is a dagger monoidal equivalence with $\mathbf{C}(K,K)$ isomorphic to $\mathbb{R}$ or $\mathbb{C}$.
\end{theorem}

Although this is not explicitly observed in \cite{Heunen2022},
  an easy consequence of Theorem~\ref{hktheorem} is that if
  $\mathbb{K}$ is $\mathbb{R}$ or $\mathbb{C}$ then
  $\mathrm{Hilb}_{\mathbb{K}}$ has a unique dagger monoidal structure
  with $K$ as unit. 

In fact it’s not hard to prove the following directly:
\begin{itemize}
\item[(1)] Any monoidal structure on $\mathbf{Hilb}_{\mathbb{K}}$ must have a unit $I$ with dimension 0 or 1.
\item[(2)] If $I=K$, the tensor product $\boxtimes$ is a dagger functor, and $r\colon K\to K\boxtimes K$ is unitary, then there are dagger monomorphisms $m_{A,B}\colon A\otimes B \to A\boxtimes B$ (where $\otimes$ denotes the usual tensor).
\item[(3)] These maps $m_{A,B}$ are invertible if and only if $K$ is a monoidal separator with respect to $\otimes$. Thus the monoidal structure is uniquely determined (up to dagger isomorphism) by these requirements. 
\end{itemize}

\appendix
\section{Real, Complex and Quaternionic Hilbert Spaces}

We record here a way to regard complex and quaternionic Hilbert spaces
as real Hilbert spaces with extra structure. We make no particular
claims for originality for the material in this section.

To give a complex vector space is equivalent to giving a real vector
space $V$ equipped with a linear transformation $s\colon V\to V$
satisfying $s^2=-1$. To give an inner product $\langle ~,~\rangle_\mathbb{C}$ on the complex vector space is equivalent to giving an inner product $[~,~]$ on the real vector space for which $s^\dagger = -s$ where $s^\dagger$ is defined with respect to $[~,~]$. Then 
\begin{align*}
\langle u,v\rangle_\mathbb{C} := [u,v]-[su,v]i.
\end{align*}

Similarly to give a quaternionic vector space is equivalent to giving
a real vector space $V$ equipped with linear maps $s,t\colon V\to V$
with $s^2=t^2=(st^2)=-1$. To give a quaternionic inner product
$\langle~,~\rangle_{\mathbb{H}}$ is equivalent to giving a real inner
product $[~,~]$ with respect to which $s^\dagger = -s$ and $t^\dagger=
-t$; then 
  $(st)^{\dagger}=t^{\dagger}s^{\dagger}=(-t)(-s)=ts=-st$ and 
\begin{align*}
\langle u,v\rangle_\mathbb{H} := [u,v]-[su,v]i-[tu,v]j-[tsu,v]k.
\end{align*}

\begin{lemma}\label{innerproductlemma}
For a complex vector space $V$ we have $u\perp su$ for all $u\in V$. Likewise if $V$ is quaternionic then $u\perp su$, $u\perp tu$ and $u\perp (stu)$.
\end{lemma}
\begin{proof}
It follows from the symmetry of the real inner product that, 
\begin{align*}
[su,u]&=[u,s^\dagger u]&&\text{(by adjoint)}\\
&=[u,-su] &&\text{(since $s^\dagger=-s$)}\\
&=-[u,su] &&\text{(by linearity)}\\
&=-[su,u] &&\text{(by symmetry)}
\end{align*}
and so $[su,u]=0$. A similar argument can be made for $t$ and $st$.
\end{proof}

\begin{lemma}\label{RCH_BDD}Let $\mathbb{K}=\mathbb{C}$ or $\mathbb{H}$, and $(V,[~,~])$ be a $\mathbb{K}$-inner product space. Then for each $x\in V$,
\begin{align*}
[x,x] = \langle x,x\rangle_\mathbb{K}.
\end{align*}
\end{lemma}

\begin{proof}
We will only show the quaternionic case. It follows directly from Lemma \ref{innerproductlemma} that for each $x\in V$,
\begin{align*}
\langle x,x\rangle_\mathbb{H} 
&=  [x,x]-[sx,x]i-[tx,x]j-[tsx,x]k\\
&= [x,x] - 0i-0j-0k\\
&= [x,x]. \qedhere
\end{align*} 
\end{proof}

This is a very convenient result and means that any statements about boundedness that apply to real inner product spaces can also apply to complex or quaternionic inner product spaces. This includes statements about completeness and so applies to Hilbert spaces. In particular, this allows us to deduce the quaternionic case in the following theorem.

\begin{theorem}\label{adjointtoboundedness}
Let $K=\mathbb{R}$, $\mathbb{C}$, or $\mathbb{H}$, and $T\colon H\to H'$ be a linear map between $K$-Hilbert spaces. If $T$ has an adjoint $T^\dagger$ then $T$ is bounded.
\end{theorem}
\begin{proof}
Let $(x_n)$ be a sequence in $H$ such that $x_n\to x$ and $Tx_n\to z$ when $n\to \infty$. Then for each $y\in K$, $\langle Tx_n, y \rangle \to \langle z,y \rangle$ but also $\langle Tx_n, y \rangle =\langle x_n, T^\dagger y \rangle \to \langle x,T^\dagger y \rangle =  \langle Tx, y \rangle$. It follows from the uniqueness of limits that $Tx=z$. By Theorem 4.13-3 in \cite{Kreyszig_1978}, $T$ is a closed linear operator and then by the closed graph Theorem 4.13-2 in \cite{Kreyszig_1978} $T$ is bounded.
\end{proof}

\section{Decomposition of bounded linear operators}\label{appendixB}

This section is based on the arguments of \cite{Bottcher:2012aa}. The
real and complex cases can be found there; the quaternionic case is a
straightforward extension.

\begin{lemma}
Let $T\colon H\to H$ be a bounded linear map on a complex Hilbert space $H$. Then $T$ is a linear combination of (four) unitary linear maps. 
\end{lemma}

\begin{proof}
Let $T\colon H\to H$ be a bounded linear map. Observe that $T$ is a linear combination of self adjoint linear maps,
\begin{align*}
T= \frac{1}{2}(T+T^\dagger)+\frac{1}{2i}(iT-iT^\dagger).
\end{align*}
Now that we have established $T$ as a linear combination of self-adjoint maps, we show that a general self adjoint map is a linear combination of unitaries. Suppose $S'\colon H\to H$ is a self adjoint linear map. We can rescale $S'$ to
\begin{align*}
S:= \frac{S'}{2\|S'\|}
\end{align*}
so that $\|S\|<1$. Consider $1-S^2$; this is clearly positive and
self-adjoint.
The square root lemma (p196 Theorem VI.9 \cite{Reed2012}) tells us
that since $1-S^2$ is a positive self-adjoint linear map, there exists
a unique positive self-adjoint linear map $R$ such that $R^2 = 1-S^2$
and moreover, $R$ commutes with each $S'$ which commutes with
$1-S^2$; in particular $SR=RS$.
We can then write
\begin{align*}
S= \frac{1}{2}(S+iR)+\frac{1}{2}(S-iR).
\end{align*}
We claim that $S+iR$ and $S-iR$ are unitary linear maps. First observe that
\begin{align*}
(S+iR)(S+iR)^\dagger &= (S+iR)(S^\dagger-iR^\dagger)\\
&= (S+iR)(S-iR)\\
&= S^2+R^2+i(RS-SR)\\ 
&= 1
\end{align*} 
and similarly $(S+iR)^\dagger(S+iR) = 1 $ making $S+iR$ unitary. The same argument holds for $S-iR$. Therefore $T$ is a linear combination of unitary linear maps. 
\end{proof}

\begin{lemma}\label{complexcase}
Let $T\colon H\to H$ be a bounded linear map on an infinite dimensional real Hilbert space $H$. Then $T$ is a linear combination of orthogonal (unitary) isomorphisms. 
\end{lemma}

\begin{proof}
Let $T'\colon H\to H$ be a bounded linear map on a separable Hilbert space $H$. Rescale $T'$ so that
\begin{align*}
T=\frac{T'}{2\|T'\|}
\end{align*}
By Lemma 3.4 \cite{Bottcher:2012aa} $T$ can be written in the form 
\begin{align*}
T= I-US
\end{align*}
where $U$ is orthogonal (unitary) and $S$ is symmetric (Hermitian). Using Lemma 3.1 \cite{Bottcher:2012aa} we can write $H$ as the orthogonal sum $H_1\oplus H_1$ of two copies of an infinite dimensional $S$-invariant closed linear subspace $H_1$ of $H$ and so $S$ can be written as the direct sum\footnote{As pointed out by Halmos in the proof of Lemma 3 \cite{Halmos1952} \textit{``The underlying Hilbert space, if it is not already separable, can be expressed as a direct sum of separable, infinite dimensional subspaces invariant under the given operator. There is, therefore, no loss of generality in restricting attention to separable Hilbert spaces."}},
\begin{align*}
S=S_1\oplus S_2\colon H_1\oplus H_1\to H_1\oplus H_1.
\end{align*}
As in Lemma 4.2 \cite{Bottcher:2012aa} we can write $S_1\oplus S_2$ as
\begin{align*}
\left( \begin{matrix}
S_1&0\\
0&S_2\\ 
\end{matrix}\right)
= 
\left( \begin{matrix}
\frac{S_1+S_2}{2}&0\\
0&\frac{S_1+S_2}{2}\\ 
\end{matrix}\right)
+
\left( \begin{matrix}
\frac{S_1-S_2}{2}&0\\
0&-\frac{S_1-S_2}{2}\\ 
\end{matrix}\right).\tag{3}\label{s}
\end{align*}
Setting
\begin{align*}
A:=\frac{S_1+S_2}{2},\quad\quad B:= \frac{S_1-S_2}{2}
\end{align*}
it follows as in Lemma 4.1 \cite{Bottcher:2012aa} that the first term in \eqref{s} can be written as
\begin{align*}
\left( \begin{matrix}
A&0\\
0&A
\end{matrix}\right)
=
\frac{1}{2}\left( \begin{matrix}
A&\sqrt{I-A^2}\\
-\sqrt{I-A^2}&A
\end{matrix}\right)
+
\frac{1}{2}\left( \begin{matrix}
A&-\sqrt{I-A^2}\\
\sqrt{I-A^2}&A
\end{matrix}\right)
\end{align*}
and the second term in \eqref{s} as
\begin{align*}
\left( \begin{matrix}
B&0\\
0&-B
\end{matrix}\right)
=
\frac{1}{2}\left( \begin{matrix}
B&\sqrt{I-B^2}\\
\sqrt{I-B^2}&-B
\end{matrix}\right)
+
\frac{1}{2}\left( \begin{matrix}
B&-\sqrt{I-B^2}\\
-\sqrt{I-B^2}&-B
\end{matrix}\right)
\end{align*}
which are linear combinations of orthogonal operators and thus so is $S$ and hence so is $T = I-US$. Therefore $T'$ is a linear combinations of orthogonal operators.
\end{proof}

\begin{lemma}
Let $T\colon H\to H$ be a bounded linear map on an infinite-dimensional quaternionic Hilbert space $H$. Then $T$ is a linear combination of orthogonal (unitary) maps. 
\end{lemma}

Consider a quaternionic Hilbert space $H$ as a real Hilbert space equipped with operators $R_i, R_j\colon H\to H$ which satisfy $R_i^2 =R_j^2 = -\id_H$, $R_iR_j = -R_jR_i$ and $R_i^\dagger = -R_i$ and $R_j^\dagger = -R_j$. The $\mathbb{H}$-linear operators on $H$ are $\mathbb{R}$-linear operators which commute with $R_i$ and $R_j$. The proof for the quaternionic case is then analogous to the real case but requires each operator to commute with $R_i$ and $R_j$. 

\begin{proof}
Let $T'$ be a bounded $\mathbb{R}$-linear map. Rescale $T'$ so that
\begin{align*}
T=\frac{T'}{2\|T'\|}.
\end{align*}
Let $R$ commute with $T'$. Then 
\begin{align*}
RT=\frac{RT'}{2\|T'\|} =\frac{T'R}{2\|T'\|} = TR.
\end{align*}
Let $Q:=I-T$, then $RQ=QR$ and $RQ^*=Q^*R$. Now $Q$ is invertible since $\|T\|=1/2< 1$ and so following from Lemma 3.3 \cite{Bottcher:2012aa} there exists an orthogonal (unitary) operator $U$ such that
\begin{align*}
Q = U |Q|
\end{align*} 
where $|Q|:=\sqrt{Q^*Q}$, a symmetric (Hermitian) operator. By the square root lemma $R|Q|=|Q|R$ since $R$ commutes with $Q^*Q$. It follows that
\begin{align*}
UR = Q|Q|^{-1}R = QR|Q|^{-1} = RQ|Q|^{-1} = RU.
\end{align*}
Let $S:=|Q|$ then $S$ is self adjoint since,
\begin{align*}
(S^*)^2=S^*S^*=(S^2)^* = (Q^*Q)^* = Q^*Q^{**} = Q^* Q = S^2
\end{align*}
hence $S^*=S$. Now,
\begin{align*}
I-T = US\quad\quad\implies \quad\quad T=I-US.
\end{align*}
We have established that $R$ commutes with $U$ and $S$. Now we want to show that such an $S$ is a linear combination of orthogonal operators that also commute with $R$. We can use use the same argument as in the proof of Lemma \ref{complexcase} to construct such a decomposition for $S$, noting that the construction of Lemma 3.1 in \cite{Bottcher:2012aa} gives a decomposition which is invariant not just under $S$ but under any operator $R$ that commutes with $S$.
\end{proof}

\subsection*{Acknowledgement}

Gratitude is extended to Frank Valckenborgh for his mentorship and guidance during the preparation of the second-named author's master's thesis and the development of this paper.

\pagebreak

\bibliographystyle{amsplain}
\providecommand{\bysame}{\leavevmode\hbox to3em{\hrulefill}\thinspace}
\providecommand{\MR}{\relax\ifhmode\unskip\space\fi MR }
\providecommand{\MRhref}[2]{%
  \href{http://www.ams.org/mathscinet-getitem?mr=#1}{#2}
}
\providecommand{\href}[2]{#2}

\end{document}